\documentclass[12pt]{amsart}
\usepackage{amsmath}
\usepackage{mathtools}

\DeclarePairedDelimiter{\set}{\{}{\}}

\usepackage{amsfonts}
\usepackage{amsthm}
\usepackage{autobreak}
\usepackage[english]{babel}
\usepackage{bold-extra}
\usepackage{hyperref}
\newtheorem{definition}{Definition}
\newtheorem*{definition*}{Definition}
\newtheorem*{prop*}{Theorem}
\newtheorem{prop}[definition]{Proposition}
\newtheorem{cor}[definition]{Corollary}

\newtheorem{lemma}[definition]{Lemma}
\newtheorem{theorem}[definition]{Theorem}
\newtheorem{question}[definition]{Question}
\newtheorem{questionss}[definition]{Questions}

\newtheorem{remark}[definition]{Remark}
\numberwithin{definition}{section}
\usepackage{enumerate}

\allowdisplaybreaks

\def\D{{\mathbb{D}}}

\def\R{{\mathbb{R}}}
\def\C{{\mathbb{C}}}
\def\Q{{\mathbb{Q}}}
\def\N{{\mathbb{N}}}

\def\T{{\mathbb{T}}}
\def\UU{{\mathcal{U}}}

\def\DD{{\mathcal{D}}}

\def\AA{{\mathcal{A}}}

\def\e{\varepsilon}

\newcommand{\vp}{\varphi}

\usepackage{multicol}
\usepackage[utf8]{inputenc}
\usepackage[a4paper]{geometry}
\usepackage{amssymb,amsmath}
\usepackage{version}
\usepackage{mathtools} 
\usepackage{enumerate}
\usepackage{dsfont}
\usepackage{mathdots}
\makeatletter

\usepackage{latexsym}
\usepackage{color}

\begin{document}

	\title{Abel universal functions: boundary behaviour and Taylor polynomials}

	\author{St\'{e}phane Charpentier,  Myrto Manolaki, Konstantinos Maronikolakis}
	
	\address{St\'ephane Charpentier, Institut de Math\'ematiques, UMR 7373, Aix-Marseille
	Universite, 39 rue F. Joliot Curie, 13453 Marseille Cedex 13, France}
	\email{stephane.charpentier.1@univ-amu.fr}
	
	\address{Myrto Manolaki, UCD School of Mathematics and Statistics, University College Dublin, Belfield, Dublin 4, Ireland}
	\email{arhimidis8@yahoo.gr}
	
	\address{Konstantinos Maronikolakis, UCD School of Mathematics and Statistics, University College Dublin, Belfield, Dublin 4, Ireland}
	\email{conmaron@gmail.com}
	
	\thanks{The research conducted in this paper was funded by the Irish Research Council and the French Ministry of Foreign Affairs under the Ulysses 2022 Scheme (No 48450RA). Konstantinos Maronikolakis acknowledges financial support from the Irish Research Council through the grant [GOIPG/2020/1562].}
	

	\maketitle
	


	\begin{abstract}
	A holomorphic function $f$ on the unit disc $\D$ belongs to the class $\mathcal{U}_A (\mathbb{D})$ of Abel universal functions if the family $\{f_r: 0\leq r<1\}$ of its dilates $f_r(z):=f(rz)$ is dense in the space of continuous functions on $K$, for any proper compact subset $K$ of the unit circle. It has been recently shown that $\mathcal{U}_A (\mathbb{D})$ is a dense $G_{\delta}$ subset of the space of holomorphic functions on $\mathbb{D}$ endowed with the topology of local uniform convergence. In this paper, we develop further the theory of universal radial approximation by investigating the boundary behaviour of functions in $\mathcal{U}_A (\mathbb{D})$ (local growth, existence of Picard points and asymptotic values) and the convergence properties of their Taylor polynomials outside $\mathbb{D}$.
		\end{abstract}
	
\section{Introduction}	
	Let $H(\mathbb{D})$ be the space of holomorphic functions on the unit disc $\mathbb{D}$ endowed with the topology of local uniform convergence. It is known that most functions in $H(\mathbb{D})$ have maximal cluster sets along any radius. More specifically, Kierst and Szpilrajn in \cite{KierstSzpilrajn1933} showed that the set of all functions $f$ in $H(\mathbb{D})$ with the property that, for any $\zeta$ in the unit circle $\mathbb{T}$ and any $w\in\mathbb{C}$, there exists a sequence $(r_n)$ in $[0,1)$ converging to $1$ such that $\lim_{n\to\infty}f(r_n\zeta)=w$, is residual in $H(\mathbb{D})$ (that is, it contains a dense $G_{\delta}$ set). An analogue of this result holds if we replace radii by paths with endpoints on the unit circle. The existence of such functions was established (among more general results) in \cite{BoivinGauthierParamonov2002} and the residuality in \cite{BernalGonzalezCalderonMorenoPradoBassas2004} (see also 
\cite{BernalGonzalezCalderonMorenoPradoBassas2005,BernalGonzalezBonillaCalderonMorenoPradoBassas2007,
BernalGonzalezBonillaCalderonMorenoPradoBassas2009} for properties of such functions). Another variant of the result of Kierst and Szpilrajn was shown by Bayart in \cite{Bayart2005}, who proved that the set of all functions $f$ in $H(\mathbb{D})$ with the property that, for any measurable function $\vp$ on $\mathbb{T}$ there exists a sequence $(r_n)$ in $[0,1)$ tending to $1$ such that $f(r_{n}\zeta)\to\vp (\zeta)$ as $n\to \infty$ for a.e. $\zeta$ in $\mathbb{T}$, is residual in $H(\mathbb{D})$. 

In this paper, we will focus on a class of functions in $H(\D)$ that possess an even more chaotic boundary behaviour:  

\begin{definition*}[Abel universal function]\label{def-Abel-univ-funct-bis}
A function $f$ in $H(\D)$ is called \textit{Abel universal} if it satisfies the following property: for every compact set $K$ in the unit circle $\mathbb{T}$, different from $\mathbb{T}$, and every continuous function $\vp$ on $K$, there exists a sequence $(r_n)$ in $[0,1)$ tending to $1$ such that $\sup_{\zeta\in K}|f(r_{n}\zeta)-\vp (\zeta)|\to 0$ as $n\to \infty$.
	The class of Abel universal functions will be denoted by $\UU_A(\D)$.
\end{definition*}
In \cite{Charpentier2020}, it was shown that $\mathcal{U}_A (\mathbb{D})$ is a dense $G_{\delta}$ subset of $H(\mathbb{D})$, which unifies the results in \cite{KierstSzpilrajn1933}, \cite{BoivinGauthierParamonov2002}, \cite{BernalGonzalezCalderonMorenoPradoBassas2004} and \cite{Bayart2005} that we mentioned before. In higher dimensions the situation is more delicate. Indeed, in \cite{GlobevnikStout1982}, it was shown that no holomorphic function on the unit ball $B$ of $\mathbb{C}^N$ ($N\geq 2$) can be unbounded (and hence dense) along every path with endpoint on $\partial B$. More recently, Globevnik, to address a conjecture in complex geometry from 1977, constructed a holomorphic function on $B$ which is unbounded along any path of finite length with endpoint on $\partial B$ (see \cite{Globevnik2015}). It turns out that, if we restrict to paths of finite length, we can replace unbounded by dense and get an analogue of the result in \cite{BernalGonzalezCalderonMorenoPradoBassas2004} that we mentioned before for several complex variables (see \cite{CharpentierKosinski2021}). 

There is a growing literature about properties and variants of Abel universal functions \cite{Charpentier2020,CharpentierKosinski2021,CharpentierMouze2022,Maronikolakis2022,Meyrath2022,CharpManoMaron}, but there is still much scope for investigation. Our objective is to study two different, but intimately connected, aspects of Abel universal functions: i) their behaviour as we approach the boundary through certain regions in $\mathbb{D}$, and ii) the behaviour of their Taylor polynomials outside $\mathbb{D}$.

It is easy to see that a function $f$ in $H(\D)$ is Abel universal if and only if the family $\{f_r:\,0\leq r<1\}$ of its \textit{dilates} $f_r(z):=f(rz)$, $z\in \overline{\D}$, is dense in the space $C(K)$ of complex-valued continuous functions on $K$, for any compact subset $K$ of $\T$, different from $\T$. Dilation is one of the most standard techniques used in the study of the boundary behaviour of functions in $H(\D)$ since it enables us to move from the unknown territory of the boundary to a region in $\mathbb{D}$, where our function is well defined and nicely behaved. For a comprehensive overview of dilation theory over the past century we refer the reader to the recent survey of Mashreghi \cite{MashreghiDilation2022}. The other standard technique for studying the boundary behaviour of functions in $H(\D)$ involves considering the partial sums of their Taylor expansion, which also behave nicely in $\mathbb{D}$ since they converge locally uniformly to our function. Thus, the class $\mathcal{U}_A (\mathbb{D})$ of Abel universal functions is the natural analogue of the well-studied class $\mathcal{U}_T(\mathbb{D})$ of \textit{universal Taylor series}. We say that a function $f$ in $H(\mathbb{D})$ belongs to $\mathcal{U}_T(\mathbb{D})$ if, for every compact set $K\subset\C\setminus\mathbb{D}$ with connected complement, the set $\{S_n(f): n\in \mathbb{N}\}$ of the \textit{partial sums} of its Taylor expansion about $0$, is dense in the space of functions in $C(K)$ that are holomorphic in the interior of $K$. In 1996, Nestoridis showed that $\mathcal{U}_T(\mathbb{D})$ is a dense $G_{\delta}$ subset of $H(\mathbb{D})$ \cite{Nestoridis1996}. Since then, various properties of universal Taylor series have been intensively studied, such as growth \cite{MelasNestoridisPapadoperakis1997,Melas2000,GardinerKhavinson2014}, gap structure of the Taylor expansion \cite{GehlenLuhMuller2000,MVY2006}, and boundary behaviour \cite{Gardiner2013,GardinerKhavinson2014,Gardiner2014,GardinerManolaki2016}. For a detailed account of several interesting boundary properties of functions in $\mathcal{U}_T(\mathbb{D})$ and the significance of the role of potential theory in their proofs, we refer to the survey of Gardiner \cite{Gardiner2018}. We mention below three motivating results, which show that universal Taylor series have an extremely wild boundary behaviour. Let  $f\in\mathcal{U}_T(\mathbb{D})$. Then:
\begin{itemize}
\item[(a)] $f$ satisfies a Picard-type property near each boundary point; that is, for every region of the form $D_{\zeta,r}:=\{z\in \mathbb{D}: |z-\zeta |<r \}$, where $r>0$ and $\zeta\in\T$, the image $f(D_{\zeta,r})$ is the entire complex plane $\mathbb{C}$ except possibly a point \cite[Corollary 2]{GardinerKhavinson2014}. 
\item[(b)]  For a.e. $\zeta$ in $\T$ and any open triangle $T_{\zeta}$ in $\D$ which has vertex at $\zeta$ and is symmetric about the radius $[0,\zeta]$, the image $f(T_{\zeta})$ is dense in $\mathbb{C}$. In particular, $f$ cannot have nontangential limits at a boundary set of positive measure \cite[Theorem 2]{Gardiner2014}.
\item[(c)] There is a residual subset $Z$ of $\T$ such that the image $\{f(r \zeta ): 0\leq r<1 \}$ is dense in $\mathbb{C}$ for every $\zeta\in Z$ \cite[Corollary 3]{Gardiner2014}.
\end{itemize}
We note that in part (c) (which follows from part (b)), the set $Z$ cannot be replaced by $\T$. Indeed, in \cite{Costakis2005} it is shown that given any closed and nowhere dense subset $E$ of $\T$, there is a function in $\mathcal{U}_T(\mathbb{D})$ that has finite radial limits on $E$. This shows that $\mathcal{U}_T(\mathbb{D})$ is not contained in $\mathcal{U}_A (\mathbb{D})$. The converse is also true, since all functions in $\mathcal{U}_T(\mathbb{D})$ have Ostrowski gaps \cite{GehlenLuhMuller2000,MVY2006}, which is not the case for all functions in $\mathcal{U}_A (\mathbb{D})$ \cite{Charpentier2020}. It is worth mentioning that the existence of Ostrowski gaps played a significant role in showing that universal Taylor series have the Picard-type property (a). Thus, it is natural to ask if this property holds for Abel universal functions. 

In Section \ref{Picard}, we show that the classes $\mathcal{U}_T(\mathbb{D})$ and $\mathcal{U}_A (\mathbb{D})$, although not strongly correlated, have similar boundary behaviour. In particular, we prove that such functions can have arbitrary growth on each closed set $A\subset \D$ such that $\overline{A}\cap \T$ is a singleton (see Theorem \ref{thm-growth-restricted}). We also show that this result, which is the dilation-analogue of \cite[Proposition 6]{Gardiner2014}, fails if $\overline{A}$ contains two paths in $\D$ with two different endpoints on $\mathbb{T}$ (see Proposition \ref{prop-2-radii}). As an application, we deduce that all functions in $\mathcal{U}_A (\mathbb{D})$ satisfy the Picard-type property (a) near each boundary point. Interestingly, we note that this property, as well as many other boundary properties of Abel universal functions (see Corollary \ref{coro-Abel-univ-funct-all-prop}), can be deduced using results about \textit{Valiron functions} \cite{BarthRippon2006} and the \textit{MacLane class} \cite{MacLane1963,Hornblower1971}. For example, one such boundary property is that all functions in $\mathcal{U}_A (\mathbb{D})$ admit $\infty$ as an asymptotic value. Finally, we discuss their angular boundary behaviour (see Corollary \ref{length}) and conclude the section with some open questions.

As we saw before, the chaotic behaviour of the Taylor polynomials of functions in $\mathcal{U}_T(\mathbb{D})$ outside $\mathbb{D}$ endows them with an extremely wild boundary behaviour. In Section \ref{partial-sums} we study the dual question: \textit{does the wild boundary behaviour of functions in $\mathcal{U}_A(\mathbb{D})$ cause a chaotic behaviour of their Taylor polynomials outside the disc of convergence?} Since the intersection $\mathcal{U}_T(\mathbb{D})\cap\mathcal{U}_A (\mathbb{D})$ is residual in $H(\mathbb{D})$, it is tempting to think that the partial sums $S_n(f)$ of the Taylor expansion of a function $f$ in $\mathcal{U}_A(\mathbb{D})$ will have a similar behaviour outside $\mathbb{D}$ (or at least on $\T$) with the ones of functions in $\mathcal{U}_T(\mathbb{D})$; that is, for each $\zeta$ in $\mathbb{C}\setminus\D$ (or in $\T$), the set $\{ S_n(f)(\zeta): n\in\mathbb{N}\}$ will be \textit{dense} in $\mathbb{C}$. We note that if $f$ is a function in $H(\D)$ which has a maximal cluster set along the radius through some point $\zeta\in\T$,
we can apply the formula 
$$f(r\zeta)=\sum_{k=0}^\infty r^k(1-r)S_k(f)(\zeta)  \ \ \ ( r\in [0,1))$$ 
to deduce that the set $\{S_n(f)(\zeta): n\in\mathbb{N}\}$ is \textit{unbounded}. It turns out, that for certain Abel universal functions we cannot replace \textit{unbounded} by \textit{dense}. Indeed, as it was shown in \cite{CharpentierMouze2022}, there are Abel universal functions $f$ such that $\{ S_n(f)(1): n\in\mathbb{N}\}$ is not dense. In fact, we will prove the following much stronger result:
 
 \begin{theorem}[Theorem \ref{thm-count-div-infty}, Corollary \ref{coro-div-to-infty-cap}, Corollary \ref{cor-sim-div-infty}]
\begin{itemize}
\item[]
\item[(a)] Given any countable set $E$ in $\mathbb{T}$, there exists $f\in\mathcal{U}_A(\D)$ such that $S_n(f)(\zeta)\to\infty$ as $n\to\infty$ for each $\zeta\in E$.
	\item[(b)]		There exists $f\in\mathcal{U}_A(\D)$ such that $S_n(f)(\zeta)\to\infty$ as $n\to\infty$ for a.e. $\zeta\in\partial\mathbb{D}$  (with respect to the arc length measure)
	and $S_n(f)(\zeta)\to\infty$ as $n\to\infty$ for a.e. $\zeta\in\C\setminus\overline{\D}$  (with respect to the 2-dimensional Lebesgue measure).
	\item[(c)]	
	There exists $f\in\mathcal{U}_A(\D)$ such that $(S_n(f))_n$ converges to $\infty$ locally in capacity on $\C\setminus\overline{\D}$;  that is, for every compact set $K$ in $\C\setminus\overline{\D}$ and $M>0$, we have $$\lim_{n\to\infty}\operatorname{cap}\left(\left\{z\in K:|S_n(f)(z)|\leq M\right\}\right)=0.$$
\end{itemize}
\end{theorem}
We note that the set of all such functions is of first Baire category and the proofs of (a) and (b) are constructive. For part (c), we show that there exists $f\in\mathcal{U}_A(\D)$ without Hadamard-Ostrowski gaps, which improves the result in \cite{Charpentier2020} about Ostrowski gaps.

Finally, in Section \ref{section-Further}, we discuss similar questions about the functions mentioned in the first paragraph (see Definition \ref{classes}). In addition, we propose some further developments in connection with classical function spaces.

\section{Notation and main definitions}\label{notation}
In this section we introduce the notation and main definitions that will be frequently used throughout this paper. Let us start with the standard ones. First, we will denote by $\N:=\{0,1,2,\ldots\}$ the set of non-negative integers. If $P$ is a polynomial, then $\text{deg}(P)$ will stand for the degree of $P$. The disc and the circle of centre $a$ and radius $r$ will be respectively denoted by $D(a,r)$ and $C(a,r)$. If $f$ is a function defined on $\D$ and $r\in [0,1)$, then $f_r$ will denote the dilate of $f$ defined by $f_r(z)=f(rz)$, $z\in \overline{\D}$. Also, given any $r\in \mathbb{R}$ and any set $A$ in $\mathbb{C}$, we write $rA=\{ rz:z\in A\}$.  For a power series $f(z)=\sum_ka_kz^k$ and $n\in \N$, the notation $S_n(f)$ will be used to denote the $n$-th partial sums of $f$; that is $S_n(f)(z):=\sum_{k=0}^na_kz^k,z\in\C$.

We will use the notation $m$ to denote the arclength measure on the unit circle $\mathbb{T}$ and $\lambda$ for the Lebesgue measure in $\C$ (identified with $\mathbb{R}^2$). For a compact set $K$ in $\mathbb{C}$, we denote by $C(K)$ the Banach space consisting of all functions continuous on $K$, endowed with the supremum norm $\|\cdot\|_K$. The notation $A(K)$ will be used for the\textit{ algebra of $K$}; that is the set of all functions in $C(K)$ that are holomorphic in the interior $K^{\circ}$ of $K$. 

Let us recall the definition of Abel universal functions. When not specifically mentioned, the letter $\rho$ will always stand for an increasing sequence $(r_n)_n$ in $[0,1)$ such that $r_n\to 1$.

\begin{definition}Let $\rho$ be as above. We denote by $\UU_A(\D,\rho)$ the set of all functions $f\in H(\D)$ that satisfy the following property: for any compact set $K\subset \T$, different from $\T$, and any function $\vp \in C(K)$, there exists an increasing sequence $(n_k)$ in $\N$ such that
	\[
	\|f_{r_{n_k}} - \vp\|_K \to 0,\quad k\to \infty.
	\]
\end{definition}

We recall that for any $\rho$ the set $\UU_A(\D,\rho)$ is a dense $G_{\delta}$ subset of $H(\D)$ \cite{Charpentier2020}. Moreover, it is easily seen that the set
\[
\UU_A(\D):=\bigcup_{\rho} \UU_A(\D,\rho),
\]
where the union runs over all increasing sequences $\rho \in [0,1)$ that converges to $1$, coincides with the set of functions $f\in H(\D)$ that satisfy the following property: for any $\e>0$, any compact set $K\subset \T$, different from $\T$, and any function $\vp \in C(K)$, there exists $r\in [0,1)$ such that
\[
\|f_r-\vp\|_K \leq \e.
\]
As in \cite{CharpentierMouze2022}, the terminology \emph{Abel universal function} will refer to the elements of $\UU_A(\D)$.

\medskip

Finally, let us set some technical notations that will be repeatedly used in several proofs.
\begin{itemize}
\item $(\varepsilon_n)_n$ will be a sequence of positive real numbers decreasing to $0$ with $\sum_{n=1}^{\infty}\varepsilon_n\leq 1/2$;
\item $(\vp_n)_n$ will be an enumeration of the polynomials with coefficients in $\Q +i\Q$;
\item $(K_n)_n$ will be a sequence of proper compact arcs in $\T$, such that for any compact set $K\subset \T$, different from $\T$, there exists $n\in \N$ such that $K\subset K_n$;
\item $\alpha,\beta: \N\to \N$ will be two functions such that for any $l,m\in \N$, there exists infinitely many $n\in \N$ for which $(\alpha(n),\beta(n))=(l,m)$.
\end{itemize}

\section{Growth and boundary behaviour of Abel universal functions}\label{Picard}
The purpose of this section is to exhibit properties of Abel universal functions in terms of notions that are classical in function theory: the maximum modulus function, asymptotic values, normality and Picard's property. Our first result will be derived from known results that we shall briefly present together with the necessary definitions.

The maximum modulus function $M(r,f):=\max\{|f(z)|:\,|z|=r\}$, $r\in [0,1)$ and $f\in H(\D)$, is commonly used to describe the growth of holomorphic functions. The next definition is that of asymptotic value, as it was introduced by Valiron \cite{Valiron1917}, see also \cite{MacLane1963}.

\begin{definition}A function $f\in H(\D)$ admits an asymptotic value $\alpha \in \C\cup \{\infty\}$ if there exists a path $\gamma:[0,1) \to \D$ with $|\gamma(r)|\to 1$ as $r\to 1$ such that $f\circ \gamma(r)\to \alpha$ as $r\to 1$.

We say that $f$ admits an asymptotic value $\alpha \in \C\cup \{\infty\}$ \textit{at a point} $\zeta \in \T$ if there exists a path $\gamma:[0,1)\to \D$ with $\gamma(r)\to \zeta$ as $r\to 1$ such that $f\circ \gamma (r)\to \alpha$ as $r\to 1$.
\end{definition}

Interestingly, the existence of asymptotic values for unbounded holomorphic functions can be related to their behaviour along spirals, where by \textit{spiral} we mean a subset of $\D$ of the form $\{r(t)e^{i\theta(t)}:\,t\in (0,+\infty)\}$, with $r(t)\to 1$ and $\theta(t)\to +\infty$ or $-\infty$ as $t\to +\infty$. Indeed, a result by Valiron asserts that an unbounded holomorphic function in $\D$ that is bounded on a spiral must admit $\infty$ as an asymptotic value \cite{Valiron1934}. These functions, that we will refer to as \textit{Valiron functions}, like in \cite{BarthRippon2006}, have been thoroughly studied, for example in \cite{BarthRippon2006,Miller1970,Warren1970}. Note that by \cite[p. 71]{MacLane1963}, there exist functions in $H(\D)$ that are unbounded near each point of $\T$ but that do not admit $\infty$ as an asymptotic value. By Valiron's result, they are of course bounded on no spirals in $\D$.

A link between the rate of growth of the maximum modulus function and the asymptotic values of functions  holomorphic in $\D$ can be made in studying the MacLane class $\AA$ of all functions in $H(\D)$ admitting an asymptotic value at any point of a dense subset of $\T$ \cite{MacLane1963}. This link is specified by the following result:

\begin{theorem}[Hornblower, \cite{Hornblower1971}]\label{Hornblower} If $f\in H(\D)$ satisfies the condition
	\[
	\int_0^1 \log^+\log^+M(r,f) dr<\infty,
	\]
	then $f$ belongs to $\AA$.
\end{theorem}

It is not difficult to see that Valiron functions cannot belong to the MacLane class $\AA$ \cite{BarthRippon2006}. Therefore, if $f\in H(\D)$ is a Valiron function, then
\[
\int_0^1 \log^+\log^+M(r,f) dr=\infty.
\]

To state our results we need to introduce some definitions.

\begin{definition}{\rm Let $f$ be holomorphic in $\D$. Then $f$ is said to be \textit{normal} if
\[
\sup _{z\in \D} (1-|z|^2)f^{\#}(z)<\infty,
\]
where $f^{\#}$ denotes the spherical derivative of $f$, defined by $f^{\#}(z)=\dfrac{|f'(z)|}{1+|f(z)|^2}$.
}
\end{definition}

The term \emph{normal} comes from Marty's normality criterion \cite[Theorem 17, p. 226]{Ahlforsbook} that asserts that $f\in H(\D)$ is normal if and only if the family $\{f\circ \vp:\, \vp\text{ automorphism of }\D\text{ onto itself}\}$ is normal. We refer to \cite[Section 9.1]{Pommerenke-univalent} for some classical results about normal functions. By a conformal mapping, this definition can be extended to any simply connected domain. More generally, we will say that a function $f\in H(\D)$ is normal at no boundary point $\zeta \in \T$ if $f$ is normal in no domain of the form $D(\zeta,r)\cap\D$, $0<r<1$, $\zeta \in \T$; that is
\[
\sup _{D(\zeta,r)\cap\D} (1-|z|^2)f^{\#}(z)=\infty,\quad 0< r <1,\,\zeta \in \T.
\]
By Montel's theorem, a holomorphic function in a domain $D$ that omits two values in $\C$, is normal in $D$. Therefore, a function $f$ that is normal at no boundary point $\zeta \in \T$ will omit at most one value in $\C$ in each set of the form  $D(\zeta,r)\cap\D$, $0<r<1$. This leads to the following definition:

\begin{definition}{\rm
		Let $f$ be a holomorphic function in $\mathbb{D}$ and $\zeta$ in $\mathbb{T}$. We say that $\zeta$ is a \textit{Picard point} of $f$ if, for each $r>0$, the set $f(D(\zeta, r)\cap \mathbb{D})$ is the whole complex plane $\mathbb{C}$ except at most one point.
	}
\end{definition}

Observe that if $\zeta\in\T$ is a Picard point of $f$ then the following seemingly stronger condition holds: $f$ assumes every complex value except possibly one \textit{infinitely often} in $D(\zeta, r)\cap \mathbb{D}$ for every $r>0$.

In \cite{BarthRippon2006} the authors remark that a Valiron function is normal at no point $\zeta \in \T$ and thus that every point in $\T$ is a Picard point of such a function. This also implies that any Valiron function that does not vanish in $\D$ admits $0$ as an asymptotic value, see \cite[p. 817]{BarthRippon2006}. The quite nice Theorems 1.2 and 1.3 in \cite{BarthRippon2006} allow us to infer that a large category of unbounded functions in $H(\D)$ enjoys a wild boundary behaviour with respect to the notions introduced above. For a subset $E$ of $\D$, we denote by $E_{NT}$ the set of all $\zeta\in\T$ with the property that there exists a sequence $(z_n)_n$ in $E$ with $z_n\to\zeta$ nontangentially.

\begin{theorem}[Barth and Rippon, \cite{BarthRippon2006}]\label{thm-BarthRippon-cons}Let $f \in H(\D)$ be bounded on a non-empty subset $E$ of $\D$. Then:
	\begin{enumerate}[(a)]
		\item if $f$ is unbounded in $\D$ and $m(\T\setminus E_{NT})=0$, then $f$ admits $\infty$ as an asymptotic value;
		\item if $f$ is unbounded near each point of $\T$ and $\T\setminus E_{NT}$ is nowhere dense in $\T$, then
		\begin{enumerate}[(i)]
			\item $f$ is normal at no point of $\T$;
			\item every point of $\T$ is a Picard point for $f$.
		\end{enumerate}
	\end{enumerate}
\end{theorem}

This theorem is very helpful for our purpose, which is to describe the boundary behaviour of Abel universal functions. Indeed, it is easily seen, by the definition, that for any $f\in \UU_A(\D,\rho)$, there exists $E\subset \D$ on which $f$ is bounded and such that $E_{NT}=\T$. In addition, it is clear that Abel universal functions cannot belong to the MacLane class $\AA$. Thus, by combining the above, we can deduce the following corollary:

\begin{cor}\label{coro-Abel-univ-funct-all-prop} Any function $f\in \UU_A(\D,\rho)$ satisfies the following properties:
\begin{enumerate}[(a)]
	\item 
	$\int_0^1 \log^+\log^+M(r,f)dr=\infty$;
	\item 
	$f$ admits $\infty$ as an asymptotic value;
	\item 
	$f$ is normal at no point of $\T$;
	\item 
	every point of $\T$ is a Picard point for $f$.
\end{enumerate}
\end{cor}

\begin{remark}{\rm (i) Another proof of part (a) can be derived from ideas originally contained in \cite{MelasNestoridisPapadoperakis1997,Melas2000}, see \cite[Corollary 2.10]{Charpentier2020}.
		
(ii) Part (d) tells us that for any $\zeta\in\mathbb{T}$ and any $0<r<1$, the set  $f(D(\zeta, r)\cap \mathbb{D})$ is the whole complex plane $\mathbb{C}$ except at most one point. This result is optimal in the sense that for any $a\in \C$ there exists $f\in \UU_A(\D,\rho)$ such that $a\notin f(\D)$. This comes from the following result in \cite{CharpManoMaron}: $e^f$ is Abel universal whenever $f$ is, hence there are Abel universal functions that do not vanish in $\D$. It is obvious that, for any $a\in \C$, $f+a$ is Abel universal if $f$ is.

(iii) For any $a\in \C$, there exists $f\in \UU_A(\D,\rho)$ for which $a$ is an asymptotic value. This is a consequence of the fact that for any function $f\in \UU_A(\D,\rho)$ that does not vanish in $\D$, and for any $a\in \C$, the function $a+1/f$ is always Abel universal (see \cite{CharpManoMaron}).

(iv) Properties (a), (b) and (d) are also valid for the class of universal Taylor series. }
\end{remark}

Finally, we can use \cite{GardinerManolaki2021} to show that Abel universal functions have chaotic angular behaviour. Let us recall that the Stolz angle $S_{\zeta}(\alpha)$ with vertex $\zeta \in \T$ and opening $\alpha>1$ is defined by
\[
S_{\zeta}(\alpha):=\set*{z\in \D:\,|z-\zeta|<\alpha(1-|z|)}.
\]
In the next result we use $\lambda_3$ to denote the Lebesgue measure on $\R^3$.
\begin{cor} \label{length}
	Let $f\in \UU_{A}(\mathbb{D},\rho)$. For $m$-a.e. point $\zeta \in \T$ and for every Stolz angle $S=S_{\zeta}(\alpha)$ with vertex $\zeta $,
	\begin{itemize}
		\item[(a)] for $\lambda _{3}$-almost every $(w,r)\in \mathbb{C}\times (0,\infty)$,
		\begin{equation*}
		\int_{S\cap f^{-1}(C(w,r))}\left\vert f^{\prime }(z)\right\vert \left\vert
		dz\right\vert =\infty ;
		\end{equation*}
		\item[(b)] for every $(w,r)\in \mathbb{C}\times (0,\infty)$,
		\begin{equation*}
		\int_{S\cap f^{-1}(D(w,r))}\left\vert f^{\prime
		}\right\vert ^{2}d\lambda=\infty .
		\end{equation*}
	\end{itemize}
\end{cor}	

The integral in Corollary \ref{length} (a) measures the total arc length of the image of $S\cap f^{-1}(C(w,r))$ under $f$, taking into account multiplicities. Thus, although this image is contained in the circle $C(w,r)$, this integral condition tells us that its length, counting multiplicities, is infinite for almost every choice of $(w,r)\in \mathbb{C}\times (0,\infty)$. Similarly, the integral condition in Corollary \ref{length} (b) tells us that the total area of the image of $S\cap f^{-1}(D(w,r))$ under $f$, counting multiplicities, is infinite for every choice of $(w,r)\in \mathbb{C}\times (0,\infty)$.
	
\begin{proof}[Proof of Corollary \ref{length}] Part (a) follows immediately from combining the fact that Abel universal functions do not have nontangential limit at any point of $\mathbb{T}$ together with \cite[Theorem 1]{GardinerManolaki2021}. Part (b) follows from part (a) and the co-area formula (see p. 3 in \cite{GardinerManolaki2021}).
\end{proof}

\medskip

It is natural to ask whether an analogue of condition (a) in Corollary \ref{coro-Abel-univ-funct-all-prop} would hold if we restrict to a smaller boundary region near a point in $\T$. For a set $A\subset \D$ such that $\bar{A}\cap \T \neq \emptyset$, let $M_A(f,r):=\sup\{|f(z)|:\,z\in A,\, |z|=r\}$ be the maximum modulus function of $f$ restricted to $A$. Condition (a) in Corollary \ref{coro-Abel-univ-funct-all-prop} tells us that $M_{\D}(f,r)$ grows rather fast to $\infty$ as $r\to 1$ for any Abel universal function $f$. One can expect that the smaller $\bar{A}\cap \T$ is, the weaker the constraints are on the rate of growth of $M_A(f,r)$. The next two results give some interesting information in this direction. The first one shows that, given any set $A\subset \D$ such that $\overline{A}\cap \T$ is a singleton, there exists an Abel universal function whose maximum modulus function restricted to $A$ has an arbitrary rate of growth.

\begin{theorem}\label{thm-growth-restricted}Let $w:[0,1)\to [1,\infty)$ be continuous and increasing such that $w(r)\to \infty$ as $r\to 1$. For any $A\subset \D$ with $\overline{A}\cap \T$ containing exactly one point, there exists $f\in \UU_A(\D,\rho)$ such that $|f(z)|\leq w(|z|)$ for any $z\in A$.
\end{theorem}

\begin{proof} Let $(\varepsilon _n)_{n}$,  $(\vp _n)_n$, $(K_n)$, $\alpha$ and $\beta$ be as in Section \ref{notation}.
Without loss of generality, we may and shall assume that $\overline{A}\cap \T=\{1\}$ and that for any $0\leq r <1$, the set $r\overline{\D}\cup \overline{A}$ is compact and has connected complement. Then for any closed arc $I$ in $\T$, the set $r\overline{\D}\cup \overline{A}\cup I$ is also compact and has connected complement.
	
	We define by induction a sequence $(P_n)_n$ of polynomials and two increasing sequences $(u_n)_n$ and $(v_n)_n$ in $\N$, as follows. We set $P_0\equiv0$ and $u_0=v_0=0$. For the inductive step, let us assume that $(u_0,v_0,P_0),\ldots,(u_{n-1},v_{n-1},P_{n-1})$ have been built. Since $w(r)\to \infty$ as $r\to 1$ and $w\geq1$, we can choose $u_{n}>u_{n-1}$ such that
	\begin{equation}\label{eqone}
	\left\vert z^{u_{n}} \left(\vp_{\alpha(n)}(1)-\sum_{k=0}^{n-1}P_k(1)\right) \right\vert \leq \varepsilon_{n}w(|z|),\quad z\in A,
	\end{equation}
	and
	\begin{equation}\label{eqtwo}
	\left\vert z^{u_{n}} \left(\vp_{\alpha(n)}(1)-\sum_{k=0}^{n-1}P_k(1)\right) \right\vert \leq \varepsilon_{n},\quad z\in r_{{v_{n-1}}}\overline{\D}.
	\end{equation}
	Then we set
	\[
	q_{n}(z):=\left\{\begin{array}{lll}
	z^{-u_{n}} \left(\vp_{\alpha(n)}(z)-\sum_{k=0}^{n-1}P_k(z)\right) & \text{if }z\in K_{\beta(n)}\\
	\vp_{\alpha(n)}(1)-\sum_{k=0}^{n-1}P_k(1) & \text{if }z\in \overline{A}\cup r_{{v_{n-1}}}\overline{\D}.
	\end{array}\right.
	\]
	From the hypothesis on $A$, we can apply  Mergelyan's theorem to choose a polynomial $P_{n}^*$ such that for any $z\in \overline{A}\cup r_{v_{n-1}}\overline{\D}\cup K_{\beta(n)}$,
	\begin{equation}\label{eqmerg}
	\left\vert P_{n}^*(z)-q_{n}(z) \right \vert \leq \e_{n}.
	\end{equation}
	By uniform continuity of $P_{n}^*(z)-z^{-u_{n}} \left(\vp_{\alpha(n)}(z)-\sum_{k=0}^{n-1}P_k(z)\right)$, we may choose $v_{n}> v_{n-1}$ such that for any $z\in K_{\beta(n)}$,
	\begin{multline}\label{unifqn}
	\Bigg|P_{n}^*(r_{v_{n}}z)-(r_{v_{n}}z)^{-u_{n}} \left(\vp_{\alpha(n)}(r_{v_{n}}z)-\sum_{k=0}^{n-1}P_k(r_{v_{n}}z)\right)- \left(P_{n}^*(z)-q_{n}(z) \right)\Bigg| \leq \e_{n}.
	\end{multline}
	Then, we define $P_{n}(z)=z^{u_{n}}P_{n}^*(z)$, and set $f:=\sum_{n\geq 0}P_n$. Observe that by \eqref{eqone} and \eqref{eqmerg}, for any $n\geq 0$ and any $z\in A$,
	\begin{eqnarray*}
		|P_{n}(z)| & = & \left\vert z^{u_{n}}P_{n}^*(z)\right \vert\\
		& \leq & \left\vert z^{u_{n}}\left(P_{n}^*(z)-q_{n}(z)\right) \right \vert + \left\vert z^{u_{n}}\left(\vp_{\alpha(n)}(1)-\sum_{k=0}^{n-1}P_k(1)\right) \right \vert\\
		& \leq & \e_n + \e_{n}w(|z|)\\
		& \leq & 2\e_nw(|z|).
	\end{eqnarray*}
	Similarly, by \eqref{eqtwo} and \eqref{eqmerg}, we get for any given $k\in\mathbb{N}$ and $n>k$, and any $|z|\leq r_{v_k}$,
	\begin{equation}\label{eqthree}
	|P_{n}(z)| \leq 2\e_n.
	\end{equation}
	This shows that $f\in H(\D)$ and that $|f(z)|\leq w(|z|)$ for any $z\in A$ (we recall that $\sum_{n=1}^{\infty}\varepsilon_n\leq 1/2$). 
	
	It remains to check that $f\in \UU_A(\D,\rho)$. Let us fix $n,m\in \N$ and $(l_k)_k$ increasing such that $\alpha(l_k)=n$ and $\beta(l_k)=m$ for any $k\in \N$. By \eqref{unifqn}, \eqref{eqmerg} and \eqref{eqthree}, we get for any $k\in \N$ and $z \in K_m$,
	\begin{eqnarray*}\left| f(r_{v_{l_k}}z) - \vp_n(r_{v_{l_k}}z)\right| & = & \left|P_{l_k}(r_{v_{l_k}}z)-\left(\vp_n(r_{v_{l_k}}z)- \sum_{i=0}^{l_k-1}P_i(r_{v_{l_k}}z)\right) + \sum_{i\geq l_k+1}P_i(r_{v_{l_k}}z)\right|\\
		& \leq & \e_{l_k} + |P_{l_k}^*(z)-q_{l_k}(z)| + \sum_{i\geq l_k+1}\left|P_i(r_{v_{l_k}}z)\right|\\
		& \leq & 2\e_{l_k} +  2\sum_{i\geq l_k+1}\e_i,
	\end{eqnarray*}
	which goes to $0$ as $k\to \infty$. Since $\vp_n(r_{v_{l_k}}z) \to \vp_n(z)$ as $k\to \infty$ uniformly on $K_m$, we deduce that $f\in \UU_A(\D,\rho)$.
\end{proof}

Applying Theorem \ref{thm-growth-restricted} with $A=(0,1)$ and $w(r)=(1-r)^{-1/2},r\in(0,1)$, we get the existence of $f\in \UU_A(\D,\rho)$ such that
\[|f(r)|\leq \frac{1}{\sqrt{1-r}},\quad r\in(0,1).
\]
This implies $(1-r)|f(r)|\leq (1-r)^{1/2},r\in(0,1)$, hence $(1-r)f(r)\to 0$ as $r\to 1$ and then the function $z\to(1-z)f(z)$ does not belong to $\UU_A(\D,\rho)$. We note that for the same function no antiderivative is Abel universal. This implies the following:

\begin{cor}\label{anti} The set $\UU_A(\D,\rho)$ is not closed under multiplication by polynomials or under antiderivation.
\end{cor}

\begin{remark}{\rm (i) Theorem \ref{thm-growth-restricted} and Corollary \ref{anti} hold for universal Taylor series as well, see \cite[Proposition 6]{Gardiner2014}.
		
(ii) We also mention that it is proven in \cite{CharpentierMouze2022} that $\UU_A(\D,\rho)$ is not closed under derivation either. The corresponding question remains open for universal Taylor series. }
\end{remark}	

Our next result shows that Theorem \ref{thm-growth-restricted} fails whenever $A$ contains two paths with two different endpoints in $\T$.

\begin{prop}\label{prop-2-radii}
	Let $\gamma_1,\gamma_2:[0,1)\to \D$ be two disjoint paths with $\gamma_1(r)\to \zeta_1$ and $\gamma_2(r)\to \zeta_2\in\mathbb{T}$ as $r\to 1$, for some $\zeta_1\neq\zeta_2$ in $\T$. Let also $h$ be a positive harmonic function on $\D$. Then there is no $f\in\UU_A(\D,\rho)$ such that $|f|\leq e^h$ on $\gamma_1\cup \gamma_2$. 
\end{prop}
\begin{proof}
	It is not difficult to see that it suffices to prove the proposition when $\gamma_1$ and $\gamma_2$ are radii, \emph{i.e.} $\gamma_i=\set*{r\zeta_i:\,0\leq r<1}$, $i=1,2$. Let us then assume that we are in this situation. For the sake of contradiction, suppose that there exists $f\in \UU_A(\D,\rho)$ such that $|f(r\zeta_i)|\leq e^{h(r\zeta_i)}$ for any $0<r<1$, $i=1,2$. Let $I$ denote a closed subarc of $\T$ having $\zeta_1$ and $\zeta_2$ as its endpoints. Since $f\in\UU_A(\D,\rho)$, there exists an increasing sequence $(\lambda_n)_n$ of integers such that $|f(r_{\lambda_n}\zeta)|\leq 1$. Hence $\log|f(r_{\lambda_n}\zeta)|\leq0$ for every $\zeta\in I$ and $n\in\mathbb{N}$. Thus, since $h$ is positive, the function $s:=\log|f|-h$ is negative on $r_{\lambda_n}I$ for any $n\in \N$ and by assumption it is also non-positive on $\gamma_1\cup \gamma_2$. Note that $s$ is subharmonic in $\D$, so by the maximum principle, $s\leq0$ on $\set*{r\zeta:0\leq r\leq r_{\lambda_n}\text{ and }\zeta\in I}$. Since $r_{\lambda_n}\to1$ as $n\to\infty$, we get that $s\leq0$ and thus $|f|\leq e^h$ on the sector $\set*{r\zeta:0\leq r< 1\text{ and }\zeta\in I}$. But, by Fatou's Theorem, positive harmonic functions on $\D$ have (finite) radial limits almost everywhere on $\T$, which contradicts the fact that $f\in\UU_A(\D,\rho)$.
\end{proof}

Using the previous proposition, we can give an alternative proof of the local Picard-type property of Abel universal functions.
\begin{proof}[An alternative proof of Corollary \ref{coro-Abel-univ-funct-all-prop} (d)]
	For the sake of contradiction, assume that there exists a function $f\in\UU_A(\D,\rho)$ that omits two values on $D(\zeta,r)\cap\D$ for some $\zeta\in\T$ and $r>0$. Without loss of generality, we can assume that it omits the values 0 and 1 on $D(\zeta,r)\cap\D$. Using Schottky’s theorem, as in the proof of \cite[Corollary 2]{GardinerKhavinson2014}, we deduce that $|f(z)|\leq e^\frac{c}{1-|z|}$ for $z\in D(\zeta,r/3)\cap\D$, where $c$ is a positive constant. Let
	\[
	P(z,\zeta)=\frac{1-|z|^2}{|\zeta-z|^2},\quad z\in\D,\zeta\in\T,
	\]
	denote the Poisson kernel and fix $\zeta_1,\zeta_2\in D(\zeta,\frac{r}{3})\cap\T$, $\zeta_1\neq \zeta_2$. Then, for $i=1,2$ and $0\leq s<1$, we have
	\[
	P(s\zeta_i,\zeta_i)=\frac{1-|s\zeta_i|^2}{|\zeta_i-s\zeta_i|^2}=\frac{1-s^2}{(1-s)^2}=\frac{1+s}{1-s}\geq\frac{1}{1-s}.
	\]
	We conclude that $P(z,\zeta_i)\geq (1-|z|)^{-1}$ for $z\in [0,\zeta_i),i=1,2$, and thus, if we set $h(z)=c(P(z,\zeta_1)+P(z,\zeta_2)),z\in\D$, we have $h(z)\geq (1-|z|)^{-1}$ for $z\in [0,\zeta_1)\cup [0,\zeta_2)$. Then
	\[
	|f(z)|\leq e^\frac{c}{1-|z|}\leq e^{h(z)},\quad z\in([0,\zeta_1)\cup [0,\zeta_2))\cap D(\zeta,r/3).
	\]
	Since $h$ is a positive harmonic function on $\D$, this contradicts Proposition \ref{prop-2-radii}.
\end{proof}

We conclude this section by three open questions that naturally arise from the previous results.
	
\begin{questionss}{\rm (i) Is condition (a) in Corollary \ref{coro-Abel-univ-funct-all-prop} 
sharp for Abel universal functions, in the sense that given any increasing function $w:[0,1)\to [1,\infty)$ that satisfies $\int_0^1 \log^+\log^+w(r)dr=\infty$, there exists $f\in \UU_A(\D,\rho)$ such that $M(r,f)\leq w(r)$?
	
		
(ii) To prove the local Picard property for universal Taylor series, Gardiner and Khavinson used their Ostrowski gap structure to show that they have very strong growth properties at every boundary point (see \cite[Theorem 1]{GardinerKhavinson2014}). In particular, they showed that if $\psi: [0,1)\to (0,\infty)$ is an increasing function such that $\int_{0}^{1}\log ^{+} \log ^{+} \psi (t) dt < \infty$,
and $f$ is a function in $H(\mathbb{D})$ that satisfies $|f(z)|\leq \psi(|z|)$ on $D(\zeta,r)\cap\mathbb{D}$ for some $\zeta\in\T$ and $r>0$, then $f\not\in\mathcal{U}_{T}$. Is this true if we replace $\mathcal{U}_{T}$ by $\mathcal{U}_{A}$?
		
(iii) Functions satisfying the assumptions of Theorem \ref{thm-BarthRippon-cons} may not be Valiron functions. Yet, is every Abel universal function a Valiron function?
	}
\end{questionss}

\section{Taylor expansion of Abel universal functions}\label{partial-sums}
	
In this section, we investigate properties of the Taylor polynomials of Abel universal functions. Specifically, we will focus on their possible gap structures and the behaviour of their Taylor partial sums outside $\D$. Our general aim is to understand to what extent being an Abel universal function affects the Taylor expansion with respect to these aspects, and \emph{vice-versa}. Before giving motivations and stating our results, we shall recall the definitions of three notions of gaps structure.

\begin{definition}{\rm Let $f(z)=\sum_{k=0}^{\infty}a_kz^{k}$ be a power series with radius of convergence $1$. We say that:
	\begin{enumerate}
	\item[(a)] $f$ possesses \emph{Hadamard gaps} if there is a sequence $(n_k)_{k}$ in $\N$ with $\inf_{k} n_{k+1}/n_{k}>1$ such that $a_j=0$ whenever $j\neq n_{k}$ for all $k\in\mathbb{N}$;
		\item[(b)] $f$ possesses \emph{Hadamard-Ostrowski gaps} if there are sequences $(p_k)_{k}$ and $(q_k)_{k}$ in $\N$ with $1\leq p_1< q_1\leq p_2< q_2\leq ...$ and $\inf_{k}q_k/p_k>1$ such that
		\begin{equation*}
		\limsup_{k\in I,k\to\infty}|a_k|^{\frac{1}{k}}<1,
		\end{equation*}
		where 	$I=\bigcup_{k=0}^\infty\{p_k+1, p_k+2, \dots, q_k\}$;
		\item[(c)] $f$ possesses \emph{Ostrowski gaps} if there are sequences $(p_k)_{k}$ and $(q_k)_{k}$ in $\N$ with $1\leq p_1< q_1\leq p_2< q_2\leq ...$ and $q_k/p_k\to \infty$ such that
		\begin{equation*}
		\limsup_{k\in I,k\to\infty}|a_k|^{\frac{1}{k}}=0,
		\end{equation*}
		where $I=\bigcup_{k=0}^\infty\{p_k+1, p_k+2, \dots, q_k\}$.	
	\end{enumerate}}
\end{definition}

It is clear that if $f$ has Ostrowski gaps then it has Hadamard-Ostrowski gaps. These definitions were motivated by the work of Ostrowski, who showed that there are intimate connections between the gap structure of a Taylor series and the behaviour of its partial sums
outside the disc of convergence. It is interesting for our purposes to mention that the existence of Hadamard gaps is related to the notion of normality, Picard's property and the MacLane class, encountered in Section \ref{Picard}. The following theorem, a compilation of different results, makes this link explicit.

\begin{theorem}\label{Recall-thms-Sons-Campbell-Murai} Suppose that $f=\sum_k a_k z^k $ has Hadamard gaps. 
	\begin{enumerate}
		\item[(a)]\label{Murai1981-thm}\emph{Hwang \cite{Hwang1986}}: If $\limsup_k|a_k|=\infty$, then $f$ assumes every complex value infinitely often in each sector of the form $\{z\in \D:\, \alpha < \arg z < \beta\}$, where $0\leq\alpha < \beta \leq2\pi$.
		\item[(b)]\label{SonsCampbell-thm}\emph{Sons and Campbell \cite{SonsCampbell}}: If $\limsup_k|a_k|=\infty$, then $f$ is not normal in $\D$;
		\item[(c)]\label{Murai1983-thm}\emph{Murai \cite{Murai1983}}: $f$ belongs to the Maclane class $\AA$.
	\end{enumerate}
\end{theorem}

\begin{remark}{\rm In part (c), the contribution of Murai was to show that if $f$ has Hadamard gaps and unbounded Taylor coefficients, then $f$ admits the asymptotic value $\infty$ at \emph{every} point of $\T$. In the case where the coefficients are bounded, the conclusion follows from an older result of Paley, see \cite{Weiss1959}.}
\end{remark}

Assertions (a) and (b) in Theorem \ref{Recall-thms-Sons-Campbell-Murai}, in comparison with assertions (c) and (d) in Corollary \ref{coro-Abel-univ-funct-all-prop}, may suggest that Abel universal functions could be found among Hadamard lacunary series. However, part (c) of Theorem \ref{Recall-thms-Sons-Campbell-Murai} rules out this possibility, since Abel universal functions do not belong to the Maclane class $\AA$:

\begin{cor}There is no Abel universal function with Hadamard gaps.
\end{cor}

In contrast, since the existence of Ostrowski gaps is a more flexible constraint, it is not difficult to construct an Abel universal function that has Ostrowski gaps (we refer the reader to \cite{CharpentierMouze2015}, where an ad hoc construction of universal functions in $H(\D)$ with \emph{large} Ostrowski-type gaps is displayed). This can also be derived from the fact that every universal Taylor series possesses Ostrowski gaps \cite{GehlenLuhMuller2000,MVY2006}, the fact that both classes of Abel universal functions and universal Taylor series are residual in $H(\D)$, and the Baire category theorem. Nevertheless it is shown in \cite{Charpentier2020} that there exist Abel universal functions without Ostrowski gaps. This is not really surprising either since the existence of Ostrowski gaps for universal Taylor series comes from the approximation properties of their Taylor partial sums on sets that are non-thin at $\infty$ (see \cite{MVY2006} for details), while Abel universal functions are defined only by properties inside $\D$ near points of $\T$. This remark is a motivation to ask whether there are Abel universal functions without Hadamard-Ostrowski gaps.

Our first result gives an affirmative answer. It even tells us that, given any prescribed function $g\in H(\D)$, we can find an Abel universal function whose Taylor coefficients are larger in modulus than those of $g$.
	
\begin{theorem}\label{coef} 
For any sequence $(\gamma_k)_k$ of positive real numbers such that $\limsup_k\gamma_k^{1/k}\leq 1$, there exists a function $f(z)=\sum_{k=0}^{\infty} a_kz^k$ in $\UU_A(\D,\rho)$ with $|a_k|\geq \gamma_k$ for every $k$ large enough. In particular, there exist Abel universal functions without Hadamard-Ostrowski gaps.
\end{theorem}

	For the proof of Theorem \ref{coef}, we will need the following simple lemma, whose verification is left to the reader.
	
	\begin{lemma}\label{doublesum}If $(\gamma_k)_k$ is as in the previous theorem, then for any $0<r<1$, the series $\sum_{j\geq n}\sum_{k\geq j}\gamma_kr^k$ tends to $0$ as $n\to \infty$.
	\end{lemma}

\begin{proof}[Proof of Theorem \ref{coef}]
Let $\rho=(r_n)_n$. Let us fix a sequence $(R_n)_n$ in $[0,1)$ such that $0<R_n<r_n<R_{n+1}<r_{n+1}<1$, $n\in\N$. We build by induction sequences $(P_n)_n$ and $(Q_n)_n$ of polynomials and an increasing sequence $(u_n)_n$ of integers. We set $P_0\equiv Q_0\equiv0$ and $u_0=0$ and once we have built $P_0,\ldots ,P_{n-1}$, $Q_0,\ldots ,Q_{n-1}$ and $u_0,\ldots ,u_{n-1}$ we choose $u_n > \deg(P_{n-1})$ and apply Mergelyan's theorem to get $P_n(z)=\sum_{i\geq u_n}a_{i,n}z^i$ so that the following hold:

\begin{enumerate}[(a)]
\item \label{coefa} $\max\left\{\sum_{i\geq u_{n}}\gamma_ir_{n+1}^i,\sum_{j\geq u_n}\sum_{i\geq j}\gamma_ir_n^i,(r_n/R_{n+1})^{u_n}/(1-r_n/R_{n+1})\right\}\leq \e_n$;
\item \label{coefb} $\sup_{|z|\leq R_n} |P_n(z)| \leq \varepsilon_n$;
\item \label{coefc} $\sup_{z\in K_{\beta(n)}}|P_n(r_nz)-\left(\vp_{\alpha(n)}(z)-\sum_{0\leq j\leq n-1}(P_j+Q_j)(r_nz)\right)|\leq \e_n.$
\end{enumerate}
Note that \eqref{coefa} is possible because of Lemma \ref{doublesum}. Then we define $Q_n(z)=\sum _{i=\text{deg}(P_{n-1})+1}^{\text{deg}(P_{n})}b_{i,n}z^i$ so that:
\begin{enumerate}[(a)]
\item[(d)] \label{itemcoef1} $b_{i,n}=\gamma_i$ for any $i=\text{deg}(P_{n-1})+1,\ldots u_n-1$;
\item[(e)] \label{itemcoef2} for any $i=u_n,\ldots,\text{deg}(P_n)$, $b_{i,n}=0$ if $|\Re(a_{i,n})|\geq \gamma_i$ and $b_{i,n}=2\gamma_i$ if $|\Re(a_{i,n})|\leq \gamma_i$.
\end{enumerate}

Observe that for any $n$ and any $i$, we have $0\leq b_{i,n}\leq 2\gamma_i$ and that, if we set $a_{i,n}=0$ for any $i=\text{deg}(P_{n-1})+1,\ldots u_n-1$, then $|a_{i,n}+b_{i,n}|\geq |\Re(a_{i,n}+b_{i,n})|\geq \gamma_i$. Using \eqref{coefb} and (\hyperref[itemcoef2]{e}), we get that the function $f(z)=\sum _{n\geq 0}(P_n+Q_n)(z)=\sum _k a_k z^k$ is in $H(\D)$ and, by the previous, that $|a_k|\geq \gamma_k$ for any $k$.

The proof will be completed once we have proven that $f\in \UU_A(\D,\rho)$. Let us fix $n,m\in \N$ and let $(n_l)_l\subset \N$ be such that for any $l\in \N$, $\alpha(n_l)=n$ and $\beta(n_l)=m$. By \eqref{coefc}, for $z\in K_m$ we have
\[
\left|f(r_{n_l}z)-\vp_n(z)\right| \leq \e_{n_l} + \left|Q_{n_l}(r_{n_l}z)\right| + \left|\sum _{j\geq n_l+1} P_j(r_{n_l}z)+Q_j(r_{n_l}z)\right|.
\]
Moreover it follows from (\hyperref[itemcoef1]{d}), (\hyperref[itemcoef2]{e}) and \eqref{coefa} that
\begin{equation*}
\left|Q_{n_l}(r_{n_l}z)\right| \leq \sum_{i\geq u_{n_{l}-1}}\gamma_{i}r_{n_l}^i \leq \e_{n_l-1}.
\end{equation*}
Now, let us observe that by \eqref{coefb} and Cauchy's inequalities we have $\left|a_{i,j}\right|\leq \e_{j}R_{j}^{-i}$ for any $j\geq n_l +1$ and any $i \in \N$, whence by \eqref{coefa} and (\hyperref[itemcoef2]{e}), for any $z\in K_m$,
\begin{eqnarray*}
\left|\sum_{j\geq n_l+1} \left(P_j(r_{n_l}z)+Q_j(r_{n_l}z)\right)\right| & \leq & 
\sum_{j\geq n_l+1}\sum_{i\geq u_{j}}\e_{j}\left(\frac{r_{n_l}}{R_{j}}\right)^i+ \sum_{j\geq n_l+1}\sum_{i\geq u_{j-1}}r_{n_l}^i
\\& \leq & 
\sum_{j\geq n_l+1}\e_{j}\frac{}{}\frac{\left(r_{n_l}/R_{n_l+1}\right)^{u_{n_l}}}{1-r_{n_l}/R_{n_l+1}}+
\sum_{j\geq u_n}\sum_{i\geq j}b_ir_n^i
\\& \leq & \left(\sum_{j\geq n_l+1}\e_{j}\right) + \varepsilon_{n_l}.
\end{eqnarray*}
Altogether we get
\[
\sup_{z\in K_m}\left|f(r_{n_l}z) -\vp_n (z) \right|\rightarrow 0\quad\text{as }l\to 0,
\]
and so $f\in \UU_A(\D,\rho)$.
\end{proof}

As recalled in the introduction, the partial sums of the Taylor expansion at $0$ of any Abel universal function must be unbounded at any point in $\T$. We will now show that for some Abel universal functions, the partial sums of their Taylor expansion even go to $\infty$ almost everywhere outside $\D$. This confirms that being Abel universal does not impose on the Taylor polynomials any oscillatory behaviour (at least on a large subset of $\C\setminus \D$).

We shall now see that Theorem \ref{coef} implies that for some Abel universal functions $f$, the sequence of partial sums $(S_n(f))_n$ converges to $\infty$ \emph{locally in capacity} in $\C\setminus \overline{\D}$. This will be a consequence of a result of Kalmes, M\"uller and Nie{\ss} that we will state below. We recall that a sequence $(f_n)_n$ of Borel-measurable functions on a set $U\subset\C$ converges to $\infty$ in capacity if, for every $M>0$, we have:
\[
\lim_{n\to\infty}\text{cap}\left(\left\{z\in U:|f_n(z)|\leq M\right\}\right)=0.
\]
Moreover, if $U$ is open, we say that $(f_n)_n$ converges to $\infty$ \emph{locally} in capacity if it converges to $\infty$ in capacity on any compact subset of $U$. 

\begin{theorem}[Theorem 1.4, \cite{KalmesMullerNiess2013}]\label{thm-Kalmes}Let $f(z)=\sum_ka_kz^k$ be a Taylor series with radius of convergence $1$ and without Hadamard-Ostrowski gaps. Then $(S_n(f))_n$ converges to $\infty$ locally in capacity in $\C\setminus \overline{\D}$.
\end{theorem}
	
We deduce from Theorems \ref{coef} and \ref{thm-Kalmes} the following result.	
	
	\begin{cor}\label{coro-div-to-infty-cap}There exists $f\in\UU_A(\D,\rho)$ such that $(S_n(f))_n$ converges to $\infty$ locally in capacity in $\C\setminus \overline{\D}$. In particular, some subsequence of $(S_n(f))_n$ converges pointwise to $\infty$ in $\C\setminus \overline{\D}$, outside a set of capacity $0$.
	\end{cor}

Furthermore, one can state a measure-theoretical version of the previous corollary using ideas of Kahane and Melas \cite{KahaneMelas2001}. Indeed, the proof of \cite[Theorem 2]{KahaneMelas2001} shows that given any $f\in H(\D)$ there exists a function $g$ in the Wiener algebra $W$, of power series that are absolutely convergent in $\overline{\D}$, such that $S_n(f+g)(z)\to \infty$ for $\lambda$-a.e. $z\in \C\setminus \overline{\D}$. Since $W\subset A(\D)$ and since $f+g$ obviously belongs to $\UU_A(\D,\rho)$ whenever $f\in \UU_A(\D,\rho)$ and $g\in A(\D)$, we immediately get the following complement of Corollary \ref{coro-div-to-infty-cap}.
	
	\begin{prop}\label{prop-partial-sums-div-a-e-outside-Dbar}For any $f\in\UU_A(\D,\rho)$ there exists $g\in W$ such that $S_n(f+g)\to \infty$ $\lambda$-a.e. in $\C\setminus \overline{\D}$ while $f+g\in \UU_A(\D, \rho)$.
	\end{prop}

So far, we have proven that, contrary to what happens for universal Taylor series, being Abel universal has no effect on the gap structures of its Taylor coefficients and no influence on the behaviour of its Taylor partial sums on large subsets of $\C\setminus \overline{\D}$. It remains to see if this is still the case on the boundary $\T$ of $\D$, where the most erratic behaviour of the partial sums could be expected. The next result gives an affirmative answer. We recall that $m$ stands for the arclength measure on $\T$.

	\begin{theorem}\label{thm-partial-sums-div-a-e-T}There exists $f\in\UU_A(\D,\rho)$ such that $S_n(f)\to \infty$ $m$-a.e. on $\T$.
	\end{theorem}
	
	\begin{proof}Let $\rho=(r_n)_n$. We fix a sequence $(R_n)_n$ in $(0,1)$ such that $R_n<r_n <R_{n+1}< r_{n+1}$ for any $n\in\N$. We will build $f$ by induction as a sum $\sum _n (P_n+Q_n)$ where $P_n$ and $Q_n$ are two polynomials that will take the following forms
		\[
		P_n=\sum_{k=u_n}^{\text{deg}(P_n)}a_{k,n}z^k,\quad Q_{n}=\sum_{i=u_{n}}^{\text{deg}(P_n)}b_{k,n}z^k,
		\]
		where $(u_n)_n$ is an increasing sequence of integers satisfying some additional growth condition, that will also be defined by induction. At each step $n$, we first define $u_n$, then comes $P_n$ and third we choose $Q_n$. To start with, we set $P_0\equiv Q_0\equiv 0$ and $u_0=0$. For the inductive step, let us assume that $(u_0,P_0,Q_0),\ldots,(u_{n-1},P_{n-1},Q_{n-1})$ have been built. We shall define $u_n$, $P_n$ and $Q_n$ as follows.
		
		We choose $u_n\in \N$ such that the following conditions are satisfied.
		\begin{enumerate}[(a)]
		\item \label{itema} $u_n\geq \max\{2,\text{deg}(P_{n-1})+1\}$;
		\item \label{itemb} $\sum_{j\geq u_n}\sum_{k\geq j}k^4r_n^k\leq \varepsilon_n$;
		\item \label{itemc} $(r_n/R_{n+1})^{u_n}\leq 1-r_n/R_{n+1}$.
		\end{enumerate}
		Note that \eqref{itemb} is possible since $\sum_{j\geq n}\sum_{k\geq j}k^4r^k\to 0$ as $n\to \infty$ for any $r\in (0,1)$ by Lemma \ref{doublesum}. Then, we apply Mergelyan's theorem to define $P_n(z)=\sum_{k\geq u_n}a_{k,n}z^k$ as a polynomial so that
		\begin{enumerate}[(a)]
			\item[(d)] \label{item1-cv} $\sup_{|z|\leq R_n} |P_n(z)| \leq \varepsilon_n$;
			\item[(e)] \label{item2-univ} $\sup_{z\in K_{\beta(n)}}|P_n(r_nz)-\left(\vp_{\alpha(n)}(z)-\sum_{0\leq j\leq n-1}(P_j+Q_j)(r_nz)\right)|\leq \e_n$
		\end{enumerate}
		It remains to define $Q_n$, which is the most important step of the construction. To proceed, we will build by a finite induction $\text{deg}(P_n)-u_n +2$ complex numbers $b_{u_n-1,n}$, $b_{u_n,n},\ldots,b_{\text{deg}(P_n),n}$. First we set $b_{u_n-1,n}=0$. Let us then assume that $b_{u_n-1,n}$, $b_{u_n,n},\ldots,b_{l-1,n}$ have been chosen for some $l\in \{u_n,\ldots,\text{deg}(P_n)\}$. To define $b_{l,n}$ we proceed as follows.
		
		For any complex number $c$, we set
		\[
		G_l(c)=\left\{w\in \T:\,\left|\sum_{j=0}^{n-1}(P_j+Q_j)(w)+\sum_{k=u_n}^{l-1}(a_{k,n}+b_{k,n})w^k +a_{l,n}w^l + cw^l\right|\leq l \right\}.
		\]
		
		We observe that $G_n(c)\cap G_n(c')\neq \emptyset$ implies that $|c-c'|\leq 2l$. Observe also that there exists an integer $s_l\geq l^{2}$ and $c_1,\ldots,c_{s_l}$ in the disc $\overline{D(0,l^4)}$ such that $|c_j-c_k|>2l$ for any $j\neq k$ (recall that $l\geq u_n\geq 2$). Therefore the sets $G_l(c_1),\ldots,G_l(c_{s_l})$ are pairwise disjoint closed subsets of $\T$, which implies that, for some $i_l\in \{1,\ldots,c_{s_l}\}$,
		\[
		m\left(G_l(c_{i_l})\right)\leq \frac{m(\T)}{s_l}\leq \frac{2\pi}{l^{2}}.
		\]
		We set $b_{l,n}=c_{i_l}$ to conclude this finite induction and
		\[
		Q_n=\sum_{k=u_n-1}^{\text{deg}(P_n)}b_{k,n}z^k=\sum_{k=u_n}^{\text{deg}(P_n)}b_{k,n}z^k
		\]
		to conclude the whole induction.
		
		Now we have to check that the power series $f=\sum_n(P_n+Q_n)$ satisfies the desired properties. First note by (\hyperref[item1-cv]{d}) that the series $\sum_nP_n$ defines a function in $H(\D)$. Moreover, the construction ensures that for every $n$ and every $k$ we have $|b_{k,n}|\leq k^4$ and so $\sum_n Q_n$ also belongs to $H(\D)$. Thus $f$ is well-defined and belongs to $H(\D)$.
		
		Let us prove that for $m$-a.e. $z\in \T$, $S_N(f)(z)\to \infty$ as $N\to \infty$. Set
		\[
		I=\cup_{n\geq 1}\{u_n,u_n+1,\ldots,\text{deg}(P_n)\}.
		\]
		The construction of the $b_{k,n}$ gives us that for any $k\in I$ and any $z\in \T\setminus G_k(b_{k,n})$, we have $|S_k(f)(z)|\geq k$, where $n$ is such that $u_n\leq k \leq \text{deg}(P_n)$. Since $\sum_km\left(G_k(b_{k,n})\right)\leq \sum_k\frac{2\pi}{k^{2}}<\infty$, by the Borel-Cantelli Theorem, $m$-a.e. $z\in \T$ belongs to all but finitely many sets $\T\setminus G_k(b_{k,n})$. Thus for $m$-a.e. $z\in\T$, $S_N(g)(z)\to \infty$ as $N\to \infty$, $N\in I$. Now, for every $n\geq 1$ and every $N\in \{\text{deg}(P_n),\ldots,u_{n+1}\}$, we have $S_N(f)(z)=S_{\text{deg}(P_n)}(f)(z)$ for any $z\in \T$. Therefore $S_N(f)(z)\to \infty$ as $N\to \infty$.
		
		To finish the proof, it remains to check that $f$ belongs to $\UU_A(\D,\rho)$. Let us fix $n,m\in \N$ and let $(n_l)_l\subset \N$ be such that for any $l\in \N$, $\alpha(n_l)=n$ and $\beta(n_l)=m$. By (\hyperref[item2-univ]{e}) we have, for any $z \in K_m$,
		\begin{equation}\label{eq-Abel-univ-Sn-infty}
			\left|f(r_{n_l}z)-\vp_n(z)\right| \leq \e_{n_l} + \left|Q_{n_l}(r_{n_l}z)\right| + \left|\sum _{j\geq n_l+1} \left(P_j(r_{n_l}z)+Q_j(r_{n_l}z)\right)\right|.
		\end{equation}
		
		By the construction of $Q_n$ (and $b_{k,n}$) and the choice of $u_n$ (see \eqref{itemb}), we have for any $z\in \T$,
		\[
		\left|Q_{n_l}(r_{n_l}z)\right|\leq \sum_{k=u_{n_l}}^{\text{deg}(P_{n_l})}k^4r_{n_l}^k\leq \varepsilon_{n_l}.
		\]
		
		To deal with the third term of the right hand-side of \eqref{eq-Abel-univ-Sn-infty}, we derive first from (\hyperref[item1-cv]{d}) and Cauchy's inequalities that for any $n\geq 1$ and any $k\in \{u_n,\ldots,\text{deg}(P_n)\}$, $|a_{k,n}|\leq \varepsilon_nR_n^{-k}$.	Then, together with the definition of the $b_{k,n}$, conditions \eqref{itemb} and \eqref{itemc}, we obtain, for any $z \in K_{m}$,
		
		\begin{eqnarray*}
			\left|\sum_{j\geq n_l+1} \left(P_j(r_{n_l}z)+Q_j(r_{n_l}z)\right)\right| & \leq & 
			\sum_{j\geq n_l+1}\sum_{k\geq u_{j}}\e_{j}\left(\frac{r_{n_l}}{R_{j}}\right)^k+ \sum_{j\geq n_l+1}\sum_{k\geq u_{j}}k^4r_{n_l}^k
			\\& \leq & 
			\sum_{j\geq n_l+1}\e_{j}\frac{\left(r_{n_l}/R_{n_l+1}\right)^{u_{n_l}}}{1-r_{n_l}/R_{n_l+1}}+
			\sum_{j\geq u_{n_l}}\sum_{k\geq j}k^4r_{n_l}^k
			\\& \leq & \sum_{j\geq n_l+1}\e_{j} + \varepsilon_{n_l}.
		\end{eqnarray*}
		
		Altogether we get
		\[
		\sup_{z\in K_m}\left|f(r_{n_l}z) -\vp_n (z) \right|\rightarrow 0\quad\text{as }l\to \infty,
		\]
		and so $f\in \UU_A(\D,\rho)$.
	\end{proof}

\begin{remark}{\rm We note that the previous result allows us to see that Abel's limit theorem strongly fails if we replace the assumption ``$S_n(f)(\zeta)$ converges in $\C$" by ``$S_n(f)(\zeta)\to \infty$", in the sense that not only $f(r\zeta)$ may not tend to $\infty$, but $f([0,\zeta))$ can even be dense in $\C$.}
\end{remark}
	
	Combining Proposition \ref{prop-partial-sums-div-a-e-outside-Dbar} with Theorem \ref{thm-partial-sums-div-a-e-T} we get the following:
	
	\begin{cor}\label{cor-sim-div-infty}There exists an Abel universal function $f$ such that
			 $S_n(f) \to \infty$ $\lambda$-a.e. on $\C\setminus \overline{\D}$ and
			 $S_n(f) \to \infty$ $m$-a.e. on $\T$.
		
	\end{cor}
	
	In Theorem \ref{thm-partial-sums-div-a-e-T}, the large set of points in $\T$ at which $S_n(f)$ tends to $\infty$ is not prescribed. However, notice that given a countable set $E \in \T$ and any set $A$ of full arclength measure in $\T$, there exists a rotation $r$ (about $0$) such that $r(E)\subset A$. Indeed, for any $z \in E$, the set $\{\zeta \in \T:\,\zeta z \in A\}$ has clearly full measure so that
	\[
	m(\bigcap_{z\in E}\{\zeta \in \T:\,\zeta z \in A\})=2\pi.
	\]
	Thus if we let $r$ be the rotation (multiplication) by any element of $\cap_{z\in E}\{\zeta \in \T:\,\zeta z \in A\}$ and let $f$ be an Abel universal function such that $S_n(f)(z)\to \infty$ for every $z$ in some set $A\subset \T$ with full measure, then $g:=f\circ r$ also belongs to $\UU_A(\D)$ and $S_n(g)(z)\to \infty$ for every $z\in E$. Therefore we have the following consequence of Theorem \ref{thm-partial-sums-div-a-e-T}.

	\begin{cor}\label{thm-count-div-infty}For any countable set $E\subset \T$, there exists $f\in\UU_A(\D,\rho)$ such that $S_n(f)(z)\to \infty$ for every $z\in E$.
	\end{cor}
	
	The proof of this corollary is based on Theorem \ref{thm-partial-sums-div-a-e-T} that relies on the Borel-Cantelli lemma. However one can easily adapt the proof of Theorem \ref{thm-partial-sums-div-a-e-T} in order to get Corollary \ref{thm-count-div-infty} without making use of any probabilistic tool. It suffices to use instead the following geometric lemma (the details of the construction of the desired Abel universal function are left to the reader):
	
	\begin{lemma}\label{lem1}
		Let $w_1,\dots,w_l\in\C,z_1,\dots,z_l\in\T$ and $R>0$. There exists $b\in\mathbb{C}$ with $|b|\leq2lR$ such that $|w_m+bz_m|\geq R$ for $m=1,\dots,l$.
	\end{lemma}
	
	\begin{proof}
		Clearly, if $|w_m|\geq R$ for $m=1,\dots,l$, then we can select $b=0$ and so we may assume that at least one of the numbers $w_1,\dots,w_l$ belongs to $D(0,R)$. Without loss of generality, we assume that $|w_1|<R$. Now, if $|w_m|\geq3R$ for $m=2,\dots,l$, then we can select $b=2R$. Thus, we may assume that at least one of the numbers $w_2,\dots,w_l$ belongs to $D(0,3R)$. Continuing in this way, we notice that the worst case is when the disc $D(0,R)$ and each annulus $\{z\in\C:(2n-1)R\leq|z|<(2n+1)R\}$ for $n=1,\dots,l-1$ contain exactly one of $w_1,\dots,w_l$. In this case we can select $b=2lR$.
	\end{proof}

	\begin{remark}{\rm 
		Note that the sets of functions in $\UU_A(\D,\rho)$ with the properties described in the results of this section are dense in $H(\D)$ since they are invariant under addition with a polynomial. However, they are of first category in $H(\D)$ because they all consist of functions that are \textit{not} universal Taylor series. (Recall that a subset $F$ of a topological space $X$ is said to be of first category in $X$ if $F$ can be written as a countable union of subsets which are nowhere dense in $X$.)
	}
	\end{remark}

	The following questions remain open:
	\begin{questionss}{\rm (i) Is it true that for any Abel universal function $f$ there exists $z\in \T$ such that $\{S_n(f)(z):\,n\in \N\}$ is dense in $\C$?
			
			(ii) In the opposite direction, does there exist an Abel universal function $f$ such that $S_n(f)\to \infty$ everywhere on $\T$?
		
		(iii) If the answer in (ii) is negative, in Corollary \ref{thm-count-div-infty}, can we replace \textit{countable} by \textit{first category}?
	}
	\end{questionss}
	
\section{Further developments}\label{section-Further}

In this final section, we will focus on extensions of the results obtained in the previous parts to other types of universal functions, that are naturally related to Abel universal functions. We mentioned in the introduction three such classes of functions in $H(\D)$ with a wild boundary behaviour. Let us recall the definitions. For $f\in H(\D)$ and $\gamma:[0,1)\to\D$ a path in $\D$ with $\gamma(r)\to \zeta$ as $r\to 1$ for some $\zeta\in\T$, the cluster set $C_{\gamma}(f)$ of $f$ along $\gamma$ is defined by
\[
C_{\gamma}(f)=\left\{w\in\mathbb{C}\cup\{\infty\}:\,w=\lim_{n\to\infty}f\left(\gamma(r_n)\right) \text{for some }(r_n)_n \text{ in } [0,1)\text{ with }r_n\to1\right\}.
\]
We will also denote by $r_{\zeta}:= \{r\zeta:\,r\in [0,1)\}$ the radius through $\zeta$ and by $C_{r_{\zeta}}(f)$ the radial cluster set through the radius $r_{\zeta}$.  Finally, we say that a cluster set is maximal if it is equal to $\mathbb{C}\cup\{\infty\}$. 
\begin{definition}\label{classes}The sets $\UU_{R}(\D)$, $\UU_{\Gamma}(\D)$ and $\UU_{ae}(\D)$ consist of all functions $f$ in $H(\D)$ that satisfy the following properties respectively:
	\begin{enumerate}[(a)]
		\item $C_{r_{\zeta}}(f)$ is maximal for any $\zeta \in \T$;
		\item $C_{\gamma}(f)$ is maximal for any path $\gamma$ in $\D$ with an endpoint in $\T$;
		\item for every $m$-measurable function $\vp$ on $\T$, there exists a sequence $(r_n)_n$ in $[0,1)$ such that $f_{r_n}(\zeta) \to \vp(\zeta)$ for $m$-a.e. $\zeta \in \T$.
	\end{enumerate}
\end{definition}
We recall that $\UU_{R}(\D)$, $\UU_{\Gamma}(\D)$ and $\UU_{ae}(\D)$ were proved to be residual in \cite{KierstSzpilrajn1933}, \cite{BernalGonzalezCalderonMorenoPradoBassas2004} and \cite{Bayart2005} respectively (see also \cite{BoivinGauthierParamonov2002} for the establishment of the existence of functions in $\UU_{\Gamma}(\D)$). It is easy to see that we have the following inclusions:
\[
\UU_A(\D)\subset \UU_{\Gamma}(\D) \subset \UU_{R}(\D) \quad \text{and} \quad \UU_A(\D)\subset \UU_{ae}(\D).
\]
It was proved in \cite{CharpentierMouze2022} that the inclusion $\UU_A(\D)\subset \UU_{\Gamma}(\D)$ is strict. 

It is natural to wonder whether all functions in these three classes satisfy the same wild boundary properties with Abel universal functions. On the one hand, it is clear that the set $\UU_{\Gamma}(\D)$ and the MacLane class do not intersect. Thus functions in $\UU_A(\D)$ must grow fast near $\T$ (see Theorem \ref{Hornblower}). On the other hand, Bayart \cite{Bayart2005} in fact proved that given any growth function $w:[0,1)\to [1,\infty)$ with $w(r)\to \infty$ as $r\to 1$, there exists $f\in \UU_{ae}(\D)$ such that $M(r,f)\leq w(r)$. We also point out that, so far, it is not clear whether $\UU_R(\D)$ may intersect the MacLane class or not. Thus we cannot use Theorem \ref{Hornblower} to make conclusions about the minimal growth of $M(r,f)$ for functions $f$ in $\UU_R(\D)$.
Moreover, one can easily check that the functions in $\UU_{R}(\D)$ satisfy the assumptions of Theorem \ref{thm-BarthRippon-cons} (a) and (b) and that the functions in $\UU_{ae}(\D)$ satisfy the assumptions of Theorem \ref{thm-BarthRippon-cons} (a), but it is not clear if they satisfy the assumptions of (b). Thus, we don't have any information about a Picard-type property that functions in $\UU_{ae}(\D)$ may satisfy. 
However we can derive the following result, using tools from potential theory, as in \cite[Corollary 6]{GardinerManolaki2016UDS}.

\begin{theorem}\label{thm-Uae(D)}
	Let $f\in H(\mathbb{D})$ be such that $\{ f(rw): 0\leq r <1 \}$ is unbounded for a.e. $w\in\T$. Then for all $\zeta\in \T$ and $r>0$, the set $f(D_{\zeta, r})$ has polar complement, where $D_{\zeta, r}=\{ z\in\mathbb{D}: |z-\zeta |<r \}$. 
	In particular, this holds if $f\in\UU_{ae} (\mathbb{D})$.
\end{theorem}	

\begin{proof}
	For the sake of contradiction assume that there is $\zeta\in\T$ and $r>0$ such that $f(D_{\zeta, r})$ has non-polar complement. Then, by Myrberg's Theorem (see \cite[Theorem 5.3.8]{ArmitageGardiner2001}), there is a positive harmonic function $h$ on $f(D_{\zeta, r})$ such that $\log |z| \leq h(z)$ there. Thus, $\log |f|$ is majorized by the positive harmonic function $h\circ f$ on $D_{\zeta, r}$, and so (by Fatou's Theorem) it is nontangentially bounded above at a.e. point in $\partial D_{\zeta, r}\cap\T$, which leads to a contradiction. 
\end{proof}

All the previous results can be summarized in the following table, where we use the notation $I(f):=\int_{0}^{1}\log^{+}\log^{+}M(r,f)dr$:

\medskip

\begin{center}
\begin{tabular}{|c|c|c|c|c|}
	\cline{2-5} \cline{3-5} \cline{4-5} \cline{5-5} 
	\multicolumn{1}{c|}{} & $\mathcal{U}_{A}(\mathbb{D})$ & $\mathcal{U}_{\Gamma}(\mathbb{D})$ & $\mathcal{U}_{R}(\mathbb{D})$ & $\mathcal{U}_{ae}(\mathbb{D})$\tabularnewline
	\hline 
	{\footnotesize{}Growth} & {\footnotesize{}$I(f)=\infty$} & {\footnotesize{}$I(f)=\infty$} & {\footnotesize{}?} & {\footnotesize{}arbitrarily slow}\tabularnewline
	\hline 
	{\footnotesize{}Asymp. values} & {\footnotesize{}$\infty$} & {\footnotesize{}$\infty$} & {\footnotesize{}$\infty$} & {\footnotesize{}$\infty$}\tabularnewline
	\hline 
	{\footnotesize{}Normality} & {\footnotesize{}At no $\zeta\in \mathbb{T}$} & {\footnotesize{}At no $\zeta \in \mathbb{T}$} & {\footnotesize{}At no $\zeta \in\mathbb{T}$} & {\footnotesize{}?}\tabularnewline
	\hline 
	{\footnotesize{}Picard points} & {\footnotesize{}Every $\zeta\in\mathbb{T}$} & {\footnotesize{}Every $\zeta\in\mathbb{T}$} & {\footnotesize{}Every $\zeta\in\mathbb{T}$} & {\footnotesize{}
	\begin{tabular}{c}
	? \tabularnewline
	(Theorem \ref{thm-Uae(D)})\tabularnewline
	\end{tabular}}
	\tabularnewline
	\hline 
\end{tabular}
\end{center}

\medskip

\begin{question}{\rm
Is it possible to complete the above table?
}
\end{question}

\medskip

Another interesting generalisation of $\mathcal{U}_A(\D,\rho)$ is given by restricting the universal approximation to a fixed compact subset of $\T$. More precisely:

\begin{definition} Let $\rho=(r_n)_n$ be a sequence in $[0,1)$ such that $r_n\to 1$, and let $K$ be a compact set in $\T$, different from $\T$. We denote by $\mathcal{U}_{A} ^K(\mathbb{D},\rho)$ the set of all functions $f\in H(\D)$ for which the family of dilates $\{f_{r_{n}}:\,n\in\mathbb{N} \}$ is dense in $C(K)$.
\end{definition}

When $K$ is a non-trivial compact arc in $\T$ (different from $\T$), the functions in $\mathcal{U}_{A} ^K(\mathbb{D},\rho)$ have similar behaviour with the functions in $\UU_A(\D,\rho)$ near points of $K$, even if the two classes are different (obviously, $\UU_A(\D,\rho)\subset \mathcal{U}_{A} ^K(\mathbb{D},\rho)$). In particular, the results of this paper that hold for classical Abel universal functions carry over to the elements of the class $\UU_{A} ^K(\mathbb{D},\rho)$, up to superficial modifications depending on $K$. The most interesting case concerns compact sets $K$ that allow functions in $\UU_{A} ^K(\mathbb{D},\rho)$ to belong to classical function subspaces of $H(\D)$. This case was examined in \cite{Maronikolakis2022}, where it was proved that if $m(K)=0$, then $\UU_{A} ^K(\mathbb{D},\rho)\cap H^p$ is a dense $G_{\delta}$ subset of the Hardy space $H^p$ of the unit disc, for $1\leq p<\infty$. Similar results were obtained for the Bergman and the Dirichlet spaces.
 
It would be interesting to examine which properties such Abel universal functions in $H^p$ satisfy, related to their boundary behaviour and the behaviour of their Taylor expansion. For example, since functions in the Hardy space belong to the MacLane class, Corollary \ref{coro-Abel-univ-funct-all-prop} (a) fails for any function in $\UU_{A} ^K(\mathbb{D},\rho)\cap H^p$ (for $m(K)=0$). Nevertheless, one may wonder whether functions of this class still have a typical minimal growth. Moreover, it is known that the set of all functions in $H^p$ with the property that every $\zeta\in\T$ is a Picard point is residual \cite{BrownHansen1972}. Thus, it natural to ask whether every $\zeta\in K$ is a Picard point for every function in $\UU_{A}^K(\mathbb{D},\rho)\cap H^p$. Furthermore, regarding the results of Section \ref{partial-sums}, the following question arises: \emph{given $K\subset \T$ with $m(K)=0$, do the Taylor partial sums of the functions of $\UU_{A} ^K(\mathbb{D},\rho)\cap H^p$ behave chaotically on $K$?} We know that this is the case for most functions in $\UU_{A}^K(\mathbb{D},\rho)\cap H^p$ (see \cite{BeiseMuller2016}). We note that in the opposite direction, it is known that if $K$ is finite, there exists a function $f$ in the disc algebra $A(\mathbb{D})$ whose Taylor partial sums are \textit{universal} on $K$ (see \cite{PapachristodoulosPapadimitrakis2019} for a precise statement). Similarly, under different sufficient conditions on $K$, it is also known that $\UU_{A} ^{K}(\D, \rho)\bigcap A^p\neq\emptyset$ and $\UU_{A} ^{K}(\D, \rho)\bigcap\DD\neq\emptyset$, where $A^p$ is the Bergman space for $1\leq p<\infty$ and $\DD$ is the Dirichlet space (see \cite{Maronikolakis2022}). The analogues of the above questions can be asked for these spaces.

\smallskip

To finish let us focus on the Bloch space $\mathcal{B}$, another classical Banach space of analytic functions, of particular interest for the study of univalent functions. We recall that a function $f\in H(\D)$ belongs to the Bloch space $\mathcal{B}$ if
\[
\|f\|_{\mathcal{B}}:=\sup_{z\in\mathbb{\D}} (1-|z|^2)|f'(z)| <\infty.
\]
We introduce the following definition:

\begin{definition}
Let $K$ be a compact subset of $\T$. We denote by $\UU_{R} ^{K}(\D)$ the set of all functions $f\in H(\D)$ for which the set $C_{r_{\zeta}}(f)$ is maximal for any $\zeta \in K$.
\end{definition}
Observe that $ \UU_{A} ^{K}(\D\ \rho)\subset \UU_{R} ^{K}(\D)$ for any $\rho$ and $K\subset \mathbb{T}$.  The first part of the next theorem shows the contrast between the Bloch space and the other classical spaces mentioned above.

\begin{theorem}\label{bloch}We have the following:
	\begin{enumerate}
		\item[(a)] $\UU_{R}(\D)\bigcap \mathcal{B} =\emptyset$.
		\item[(b)] $\UU_{R} ^{\mathbb{T}\setminus E}(\D)\bigcap \mathcal{B} \neq\emptyset$, for some $E$ in $\mathbb{T}$ with $m(E)=0$.
	\end{enumerate}
\end{theorem}

\begin{proof}
(a) is a consequence of Theorem \ref{thm-BarthRippon-cons} (b) and the fact that each function in $\mathcal{B}$ is normal (see p. 71 in \cite{Pommerenke}).

(b) It is known that there is a function $f$ in $\mathcal{B}$ that has radial limits at no point of $\mathbb{T}$. For example, one can consider the function $f(z)=\sum _{n=1}^{\infty} z^{2^{n}}$ which is in $\mathcal{B}$ as a Hadamard lacunary series with bounded coefficients. By the Tauberian theorem of Hardy and Littlewood (which states that if a Hadamard lacunary power series has a radial limit at a point then the power series converges at that point), $f$ has no radial limits. We will now make use of a stronger version of Plessner's Theorem for functions in the Bloch space. More precisely, \cite[Corollary 6.15]{Pommerenke} tells us that for any function $g\in\mathcal{B}$ we have that for $m$-a.e. $\zeta \in \mathbb{T}$ either $g$ has a finite nontangential limit at $\zeta$ or the cluster set of $g$ along $r_{\zeta}$ is maximal. Thus, applying this to $f$, we get that for some set $E$ in $\mathbb{T}$ with zero arclength measure we have that $f\in\UU_{R} ^{\mathbb{T}\setminus E}(\D)\bigcap \mathcal{B}$.
\end{proof}

This result suggests the following question: 
\begin{question} In Theorem \ref{bloch} (b), can we replace $\UU_{R} ^{\mathbb{T}\setminus E}(\mathbb{D})$ by the smaller class $\UU_{A} ^{\mathbb{T}\setminus E}(\mathbb{D}, \rho)$?
\end{question}

\bibliographystyle{amsplain}
\bibliography{refs}
	
\end{document}